\documentclass[12pt]{amsart}
\usepackage{amsmath}
\usepackage{amssymb}
\usepackage{epsfig}
\usepackage{wasysym}
\usepackage{graphicx}
\usepackage{tikz}
\usepackage{tikz-cd}
\usepackage{bm}
\usepackage{xcolor}
\usepackage{listings}
\numberwithin{equation}{section}
\usepackage{extpfeil}
\usepackage{array}
\usepackage{adjustbox}
\usepackage{longtable}
\usepackage{makecell}
\usepackage[all]{xy}

\input xy
\xyoption{all}

\DeclareFontFamily{U}{mathb}{\hyphenchar\font45}
\DeclareFontShape{U}{mathb}{m}{n}{
      <5> <6> <7> <8> <9> <10> gen * mathb
      <10.95> mathb10 <12> <14.4> <17.28> <20.74> <24.88> mathb12
      }{}
\DeclareSymbolFont{mathb}{U}{mathb}{m}{n}
\DeclareMathSymbol{\righttoleftarrow}{3}{mathb}{"FD}

\calclayout
\allowdisplaybreaks[3]

\theoremstyle{plain}
\newtheorem{prop}{Proposition}[section]

\newtheorem{theo}[prop]{Theorem}

\theoremstyle{definition}

\newtheorem{rema}[prop]{Remark}

\newtheorem{exam}[prop]{Example}

\newcommand{\actsfromleft}{\mathrel{\reflectbox{$\righttoleftarrow$}}}
\newcommand{\actsfromright}{\righttoleftarrow}

\def\cE{{\mathcal E}}

\def\cU{{\mathcal U}}

\def\fA{{\mathfrak A}}

\def\fD{{\mathfrak D}}

\def\fF{{\mathfrak F}}

\def\fK{{\mathfrak K}}

\def\fS{{\mathfrak S}}

\def\Sing{\mathrm{Sing}}

\def\Pf{{\mathrm {Pf}}}

\def\fS{{\mathfrak S}}

\def\bA{{\mathbb A}}
\def\bG{{\mathbb G}}
\def\bP{{\mathbb P}}

\def\bC{{\mathbb C}}

\def\bF{{\mathbb F}}

\def\Aut{\mathrm{Aut}}
\def\Gr{\mathrm{Gr}}

\def\SL{\mathrm{SL}}

\def\sD{\mathsf{D}}
\def\sA{\mathsf{A}}

\def\Burn{\mathrm{Burn}}

\def\lim{\mathrm{lim}}

\def\Sym{\mathrm{Sym}}

\def\SL{\mathsf{SL}}
\def\PSL{\mathsf{PSL}}

\makeatother
\makeatletter

\begin{document}

\title[Stable birationalities]{Stable equivariant birationalities of cubic and degree 14 Fano threefolds}

\author{Yuri Tschinkel}
\address{
  Courant Institute,
  251 Mercer Street,
  New York, NY 10012, USA
}

\address{Simons Foundation\\
160 Fifth Avenue\\
New York, NY 10010\\
USA}

\email{tschinkel@cims.nyu.edu}

\author{Zhijia Zhang}
\address{
  Courant Institute,
  251 Mercer Street,
  New York, NY 10012, USA
}
\email{zhijia.zhang@cims.nyu.edu}

\date{\today}

\begin{abstract}
We develop an equivariant version of the Pfaffian-Grassmannian correspondence and apply it to produce
examples of nontrivial twisted equivariant stable birationalities between cubic threefolds and degree 14 Fano threefolds. 
\end{abstract}

\maketitle

\section{Introduction}
\label{sect:intro}

Over algebraically closed fields, birationality of smooth cubic threefolds 
$$
Y\subset \bP^4
$$ 
and associated Fano threefolds $X$ of degree 14 and Picard rank 1, 
$$
X=\bP^9\cap \Gr(2,6)\subset \bP^{14},
$$
is well-known classically.  
Recently, it has been reconsidered from the perspective of vector bundles and derived categories \cite{KuzCubic}. However, birationality in presence of group actions is widely open. 

In this paper, we work over an algebraically closed field $k$ of characteristic zero and focus on equivariant birationalities. 
Our starting point is the following beautiful example: 
There is a distinguished smooth cubic threefold $Y$, the {\em Klein cubic}, with $G:=\Aut(Y)=\PSL_2(\bF_{11})$; it is 
given by
$$
y_1^2y_2+y_2^2y_3+y_3^2y_4+y_4^2y_5+y_5^2y_1=0. 
$$
There is also a unique smooth Fano threefold $X$ of degree 14 and Picard number 1, admitting a regular, generically free, action of $G$. The threefolds $X$ and $Y$ are birational.
However, these $G$-actions are birationally rigid, and thus not equivariantly birational to each other \cite[Theorem 4.3]{BC-finite}. 
We complement this result by showing:

\begin{theo}[Proposition~\ref{prop:psl211}]
\label{thm:klein-stable}
Let $Y$ be the Klein cubic threefold and $X$
the associated Fano threefold of degree 14, with a generically free regular action of  
$G=\mathsf{PSL}_2(\bF_{11})$.
Then there is a $G$-equivariant 
birationality
$$
Y\times \bP^2\times \bP(V)
\sim_G
X\times \bP^2 \times \bP(V),
$$
with trivial $G$-action on $\bP^2$, and projectively linear $G$-action on the projectivization of an irreducible 6-dimensional representation $V$ of the central extension $\tilde{G}=\mathsf{SL}_2(\bF_{11})$. 
\end{theo}

Birational rigidity techniques work well when the group under consideration is {\em large}, with large orbits. We are not aware of results such as \cite[Theorem 4.3]{BC-finite} for actions of any other finite groups on cubic threefolds $Y$ and the associated $X$.  

In this note, we apply the recently developed formalism of equivariant Burnside groups \cite{BnG} to
exhibit actions failing equivariant birationality for such pairs.  
We introduce the notion of {\em twisted} equivariant stable birationality (in Section~\ref{sect:birgeo}) and produce examples of equivariantly nonbirational but twisted equivariantly stably  birational varieties, for actions of finite groups; this relies on 
vector bundle techniques, and in particular, the no-name lemma, which is ubiquitous in equivariant birational geometry. For example, we prove in Section~\ref{sect:C3D4}: 

\begin{theo}
    There exist smooth cubic threefolds $Y$, with associated birational Fano threefolds $X$ of degree 14, such that 
    \begin{itemize}
        \item $Y$ and $X$ carry a generically free regular action of
        $G=\fK_4$, the Klein four-group,
        \item the $G$-actions are not equivariantly birational, 
        \item the $G$-actions are twisted equivariantly stably  birational.
    \end{itemize}
\end{theo}

The proofs proceed via an equivariant version of the classically known Pfaffian-Grassmannian construction, recalled in Section~\ref{sect:dual}. 
We classify automorphisms of smooth cubic threefolds admitting an equivariant Pfaffian realization, in Proposition~\ref{prop:lift}.  
In Section~\ref{sect:psl2}, we explain our approach in the case of $G= \PSL_2(\bF_{11})$; we show that the recently introduced Burnside invariants {\em do not} distinguish the actions on $Y$ and $X$. In Section~\ref{sect:C3D4}, we turn to $G= C_3\rtimes \mathfrak D_4$ and its subgroup $\fK_4$. 
In this case, birational rigidity techniques fail, but the Burnside invariants show nonbirationality of the actions. In Sections~\ref{sect:cubic-sing} and \ref{sect:s5},
we consider actions of dihedral groups and of the symmetric group $\fS_5$, on 
 {\em singular} $Y$ and $X$.

\

\noindent
{\bf Acknowledgments:} 
We are grateful to Brendan Hassett for his interest, numerous suggestions, and encouragement. 
Comments by Christian B\"ohning, Hans-Christian von Bothmer, and Andrew Kresch were also very helpful.
The first author was partially supported by NSF grant 2301983. 

\section{Equivariant birational geometry}
\label{sect:birgeo}

\subsection*{Birationality and stable birationality}
We consider $G$-varieties, i.e., projective varieties $X$ with a
regular, generically free action of a linear algebraic group $G$;
in most of our applications, $G$ is a finite group. 
Equivariant birationality of $G$-varieties $Y,X$ is denoted by 
$$
Y\sim_G X.
$$
Equivariant stable birationality means that 
$$
    Y\times \bP^m\sim_G X\times \bP^m, 
$$
for some $m$, with {\em trivial} action on the second factor. 
Examples of equivariantly nonbirational but stably birational actions are sought after, see \cite[Section 4]{kollar}, and produced  in 
\cite{HT-torsor}, \cite{Boe}, \cite{Boe-2}.

These notions should be viewed as analogous to birationality and stable birationality of varieties over nonclosed fields. 

\subsection*{Twisted equivariant stable birationality}
One important distinction between equivariant geometry and geometry over nonclosed fields is that there is only {\em one} projective space of a given dimension, but possibly several, nonbirational $\bP(V)$ for a (projectively) linear action of a finite group,  see \cite{TYZ-3}. 
This leads us to the notion of {\em twisted equivariant stable birationality}, 
when 
$$
Y\times  \prod_j \bP(V_j) \sim_G X\times \prod_j \bP(V_j),
$$ 
where $V_j$ are linear representations of extensions of $G$ such that $G$ acts (projectively) linearly on $\bP(V_j)$, for all $j$.  This does not preclude equivariant stable birationality of $Y$ and $X$, a priori. See, e.g., \cite[Section 2]{HT-intersect} for a discussion of (projectively) linear actions, linearizability, and stable linearizability. 

Over nonclosed fields, this notion is analogous to birationality after multiplication with Brauer-Severi varieties.

\subsection*{No-name lemma}
Let $\mathcal E\to X$ be a $G$-vector bundle of rank $m$, with a generically free $G$-action on $X$ and $\cE$. 
    Then 
    $$
    \mathcal E\sim_{G} X\times \bP^m,
    $$
    with trivial action on $\bP^m$, see, e.g., \cite[Section 4.3]{CGR}.
In particular, for any $G$-representation $V$ of dimension $m$, one has  
$$
X\times V \sim_{G} X\times \bP^m, 
$$
with trivial action on $\bP^m$. This also implies that all faithful {\em linear} actions of a finite group $G$ on projective spaces are stably birational. 
For {\em projectively linear} actions of $G$ arising from faithful representations $V,W$ of a central extension $\tilde{G}$ of $G$, with the same central character,  
the no-name lemma applied to the extension
$$
1\to \bG_m\to G'\to G\to 1, 
$$
(where $G'$ contains $\tilde{G}$)
shows that
\begin{equation} 
\label{eqn:proj}
\bP(W)\times \bP^{m} \sim_G \bP(W\oplus V),  \quad m=\dim(V), 
\end{equation}
with trivial action on $\bP^{m}$. 
The no-name lemma substantially simplifies the birational geometry of equivariant vector bundles; we will apply it in Section~\ref{sect:dual}.

\subsection*{Equivariant Burnside groups}

This formalism produces a homomorphism from the free abelian group on {\em equivariant birational types}, i.e., equivariant birationality classes of $n$-dimensional varieties, to 
$$
\Burn_n(G),
$$
a group defined by generators
$$
\mathfrak s=(H, Z\actsfromleft k(F), \beta),
$$
subject to conjugacy and blowup relations \cite{BnG}. 
The class of a $G$-action on a variety is computed on a {\em standard model}, see \cite[Section 7.2]{HKTsmall}. On such a model $X$, all stabilizers are abelian, and the class
$$
[X\actsfromright G]:=\sum_{H,F}\,\, (H, Z\actsfromleft k(F), \beta) \in \Burn_n(G)
$$
is a sum over all (conjugacy classes of) abelian subgroups $H\subseteq G$, and strata $F\subseteq X$, of dimension $d\le n$, with generic stabilizer $H$ and residual action of $Z\subseteq Z_G(H)/H$; here $\beta=(b_1,\ldots, b_{n-d})$ is the collection of weights of $H$ in the normal bundle to $F$. 

One of the relations in $\Burn_n(G)$ states that the symbol $\mathfrak s$ vanishes if there exists a nonempty $I\subseteq [1,\ldots, n-d]$ such that $\sum_{i\in I} b_i=0$.

There is a subgroup 
$$
\Burn_n^{\rm{inc}}(G)\subset \Burn_n(G)
$$
{\em freely} generated by {\em incompressible} symbols, see, e.g., \cite[Section 3.6]{TYZ-3}. In many situations, we can distinguish $G$-actions already via projection of  $[X\actsfromright G]$ to this subgroup. 
We record: 

\begin{exam}
\label{exam:incomp}
Let $Y\subset \bP^4$ be a cubic threefold with the action of $G=\mathfrak K_4$ via \eqref{eqn:0}. Note that one of the generating involutions fixes a cubic surface $S\subset Y$, and the fixed locus of the residual involution on $S$ is a cubic curve $C\subset S$.

Assume that $C=Y^G$ is {\em smooth}. Then the symbol
$$
(C_2, C_2\actsfromleft k(S), (1))
$$
is incompressible, and the $G$-action is not linearizable, see \cite[Proposition 2.6]{CTZ}. 

Such a symbol also arises from a 
$\mathfrak K_4$-action on a 4-nodal cubic threefold \cite[Example 5.2]{CTZ}; or from a 
$C_4$-action in \cite[Example 2.7]{CTZ}.
\end{exam}

\begin{exam} 
\label{exam:dih}
Let $Y\subset \bP^4$ be a singular cubic threefold 
given by 
$$
y_1y_2y_3+f_3(y_3,y_4,y_5)=0,
$$
where $f_3$ is a cubic form. It carries the action of the dihedral group $G:=\mathfrak D_{2n}$ of order $4n\ge 8$, via
$$
y_1\leftrightarrow y_2, \quad (y_1,y_2,y_3,y_4,y_5) \mapsto (\zeta y_1,\zeta^{-1}y_2, y_3,y_4,y_5),
$$
where $\zeta$ is a primitive root of unity of order $2n$. 

Then $G$ fixes a plane cubic curve $C\subset \bP^2_{y_3,y_4,y_5}$ given by $f_3=0$. Assume that $C$ is {\em smooth}.
To reach a standard model, one has to blow up $C$. Computing the class 
$[Y\actsfromright\fD_{2n}]$ on 
such a model, we find the {\em incompressible} symbol
\begin{align}\label{eq:2d4incomp}
    (C_2,\fD_{n}\actsfromleft k(C)(t), (1)), \quad n \ge 2, 
\end{align}
where $C$ is a genus 1 curve, see \cite[Proposition 5.17]{CMTZ}. Choosing  roots of unity $\zeta,\zeta'$ such that $\zeta\neq \pm \zeta'$, we obtain equivariantly nonbirational, nonlinearizable $G$-actions on the rational cubic threefold $Y$. 
\end{exam}

\section{Dualities}
\label{sect:dual}

\subsection*{Pfaffian-Grassmannian correspondence}
We follow the presentation in \cite[Section 2]{KuzCubic}, which builds on classical constructions, see, e.g., \cite{Puts}, \cite{IM}.
Let $A$ and $V$ be vector spaces over $k$ of dimension 5, respectively, 6. Let 
$$
f:A\to \wedge^2(V^{\vee})
$$ 
be an injective linear map, called {\em $A$-net of skew forms on $V$}. It is called {\em regular}, if $\mathrm{rk}(f(a))\ge 4$, for all nonzero $a\in A$. 

We consider the following varieties, which are smooth for generic $f$: 
\begin{align*} 
Y_f & = \bP(f(A)) \cap \Gr(2,V)^\vee, \text{a cubic threefold}, \\
X_f & = \bP(f(A)^{\perp}) \cap \Gr(2,V), \text{a degree 14 Fano threefold},
\end{align*} 
where $\Gr(2,V)^\vee\subset\wedge^2(V^\vee)$ is the projective dual of $\Gr(2,V)$, given by the vanishing of the Pfaffian cubic form $\Pf\in\Sym^3(\wedge^2(V^\vee))$. Note that $\Gr(2,V)^\vee$ parametrizes skew forms on $V$ of rank at most four.
Let $\cU_f$ be the restriction of the tautological rank two bundle on $\Gr(2,V)$ to $X_f$. For a {\em regular} $f$, there is a natural rank two 
theta-bundle $\cE_f$ over $Y_f$, see \cite[Section 2]{KuzCubic} and \cite[Theorem 2.2]{IM}; 
the dual bundle $\cE_f^\vee$ is the rank two subbundle of the rank six trivial bundle $Y_f\times V$ given by 
$$
\cE_f^\vee=\{(y,v)\in Y_f\times V\mid v\in \ker(y)\}.
$$
The injective classifying morphism 
$$
\kappa: Y_f\to \Gr(2,V),\quad y\mapsto\ker(y)
$$ 
induces an embedding of $Y_f$ in $\Gr(2,V)$. Under this embedding, $\cE_f^\vee$ is naturally identified with the restriction of the tautological bundle from $\Gr(2,V)$ to $Y_f$. 
Furthermore,
\begin{equation} \label{eqn:sing}
\mathrm{Sing}(X_f)=\mathrm{Sing}(Y_f) = X_f\cap Y_f\subset \Gr(2,V).   
\end{equation}
By \cite[Theorem 2.18]{KuzCubic}, for a {\em regular} $f$, we have a diagram
\begin{equation} \label{dia:vec}
\begin{tikzcd}
\cE^\vee_f \ar[d] \ar[dr,"\psi"] \ar[rr,dashed,"\theta"]&     & \ar[dl,swap,"\phi"] \cU_f \ar[d] \\
Y_f                          & V & X_f
\end{tikzcd}
\end{equation}

 

\noindent 
where the morphisms $\psi,\phi$ are induced by 
the natural projection from the tautological bundle over $\Gr(2,V)$ to $V$.
The images of $\psi$ and $\phi$ can be described as follows
\begin{align*} 
\psi(\cE_f^\vee)& = \{v \in V \mid v \in \ker(f(a)) \text{ for some nonzero }  a \in A\},
\\
\phi(\cU_f) & = \{v \in V\mid v \in \ell \text{ for some } \ell \in X_f\subset \Gr(2,V)\}.
\end{align*}  
Linear algebra shows that 
$
\psi(\cE_f^\vee) = \phi(\cU_f),
$ see e.g., \cite[Proposition 2.11, 2.15]{KuzCubic}, \cite[Theorem B]{Puts}. In fact, the common image of $\psi$ and $\phi$ is a quartic hypersurface
$$
Q_f\subset V,
$$
singular along the affine cone $\tilde C_f$ over a curve $C_f\subset \bP(Q_f)\subset\bP(V)$.  Both $\psi$ and $\phi$ are isomorphisms on the complement $Q_f\setminus \tilde C_f$. The composition $\theta:=\phi^{-1}\circ \psi$ is then a birational map between vector bundles. After projectivization, $\theta$ induces a birational map between $\bP^1$-bundles 
$$
\bP(\cE^\vee_f)\dashrightarrow\bP(\cU_f),
$$
which is a flop in a ruled surface \cite[Theorem 2.17]{KuzCubic}.

\begin{rema}
In the literature, the existence of the diagram \eqref{dia:vec} and the birationality of $\theta$ are proved for the projectivizations $\bP(\cE^\vee_f), \bP(\cU_f)$ and $\bP(V).$ However, the underlying linear algebra proof applies to the vector bundles verbatim. To study the equivariant geometry of this construction, we use the diagram of vector bundles \eqref{dia:vec}.
\end{rema}

As explained in \cite[Remark 2.19]{KuzCubic}, fixing a hyperplane $\Pi\subset \bP(V)$, there are induced birational maps
\begin{equation}
    \label{eqn:rho}
\varrho_\Pi: Y_f\stackrel{\sim}{\dashrightarrow} \Pi\cap \bP(Q_f) \stackrel{\sim}{\dashleftarrow} X_f.
\end{equation}

\subsection*{Equivariant Pfaffian-Grassmannians}
To arrive at an action of a finite group $G$ on $Y$ and $X$, 
we start with a faithful 6-dimensional $\tilde G$-representation $V$ of a central extension 
$$
1\to \mathbb G_m\to \tilde{G}\to G\to 1
$$
which induces a generically free action of $G$ on $\bP(V)$, i.e., the central $\bG_m$ is acting trivially on $\bP(V)$.  
Let $A$ be a 5-dimensional vector space and
\begin{align}\label{eqn:Anet}
    f: A\to \wedge^2(V^\vee)
\end{align}
a regular $A$-net such that 
$$
f(A) \subset \wedge^2(V^\vee)
$$
is a 5-dimensional $\tilde{G}$-invariant subspace. Assume that the induced $G$-actions on $\bP(f(A))$ and on $\bP(f(A)^\perp)$ are also generically free. Then 
\begin{equation} 
\label{eqn:xy}
Y_f:=\Gr(2,V)^\vee \cap \bP(f(A)), \quad X_f:=\Gr(2,V)\cap \bP(f(A)^\perp), 
\end{equation}
carry $G$-actions, which are again generically free.  
The $\tilde{G}$-actions naturally lift to the vector bundles $\cE^\vee_f$ and $\cU_f.$ 

We refer to this construction as {\em equivariant} {\em Pfaffian-Grassmannian correspondence}. 
A generically free $G$-action on a smooth cubic threefold $Y$ is called {\em equivariantly Pfaffian} if it arises from an equivariant Pfaffian-Grassmannian 
correspondence. 

\begin{rema} 
Recall that every smooth cubic threefold over $\bC$ admits a Pfaffian representation \cite{AdlerRam}, \cite{MT}; in fact, this holds also for singular cubics \cite{coma}. 
By \cite[Theorem 8.2]{Beau-Pfaff}, a smooth cubic threefold $Y$ over a nonclosed field $k$ is Pfaffian if and only if there is 
an arithmetically Cohen-Macaulay curve $C\subset Y$, not contained in a hyperplane, with $K_C=\mathcal O_C$, i.e., an elliptic normal quintic, defined over $k$. 

In the equivariant context, this criterion fails: the presence of  a $G$-stable elliptic quintic $C$ on $Y$ implies equivariant birationality with the corresponding $X$, see, e.g., \cite[Theorem 1.1]{IM}. However, in Section~\ref{sect:C3D4}, we produce examples of equivariantly Pfaffian actions on smooth $Y$ and $X$, which are not equivariantly birational.   
\end{rema}

\subsection*{Twisted equivariant stable birationality}
Given an equivariant Pfaffian-Grassmannian correspondence,   
the diagram \eqref{dia:vec} constructed from a regular $A$-net \eqref{eqn:Anet} is $\tilde G$-equivariant. In particular, the corresponding birational maps $\psi,\phi$ and $\theta$ in \eqref{dia:vec} are $\tilde G$-equivariant since their constructions are canonical. We have an equivariant birationality 
$$
\cE_f^\vee  \, \sim_{\tilde{G}} \,  \cU_f. 
$$
However, the no-name lemma {\em does not} apply directly to this situation, since the $\tilde{G}$-action has a nontrivial generic stabilizer on the bases $Y_f$ and $X_f$. Therefore, we use a variant: let $W$ be a faithful representation of $\tilde{G}$, inducing a generically free $G$-action on $\bP(W)$, 
with the central $\bG_m$ acting via scalars. E.g., we could put $W=V$. 
Consider the diagram

\

\centerline{
\xymatrix{
\cE_f^\vee\times W \ar[d] \ar@{-->}[r]^{\sim} & \cU_f\times W\ar[d] \\
Y_f\times W  & X_f \times W
}
}

\ 

\noindent 
The horizontal map is a $\tilde{G}$-equivariant birational isomorphism.
The vertical maps are $\tilde{G}$-equivariant rank two vector bundles, and the action of $\tilde{G}$ on the respective bases is generically free. By the no-name lemma, we have
$$
Y_f\times \bA^2\times  W \sim_{\tilde{G}} 
X_f\times \bA^2\times W, 
$$
with trivial $\tilde{G}$-action on the $\bA^2$-factors. 
The projections to the respective bases are equivariant under the action of $\mathbb G_m=\ker(\tilde{G}\to G)$. This implies $G$-equivariant birationality of the quotients
$$
(Y_f\times \bA^2\times  W) /\mathbb G_m\sim_G
(X_f\times \bA^2\times W)/\mathbb G_m, 
$$
and therefore
$$
Y_f\times \bP^2\times \bP(W)\sim_G
X_f\times \bP^2\times \bP(W),
$$
with trivial $G$-action on $\bP^2$ and projectively linear $G$-action on $\bP(W)$.
This yields: 

\begin{prop} 
\label{prop:stable}
Given an equivariant Pfaffian-Grassmannian correspondence, we 
have a twisted equivariant stable birationality between the corresponding cubic threefold $Y_f$ and the associated degree 14 Fano threefold $X_f$.
\end{prop} 

\begin{rema} 
We do not know whether or not $Y_f$ and $X_f$ are equivariantly stably birational, unless we can apply the no-name lemma directly to $\cE_f^\vee$ and $\cU_f$, i.e., when $V$ can be chosen to be a $G$-representation. 
\end{rema}

\begin{rema}
\label{rema:birat}
If there is a $G$-stable hyperplane $\Pi\subset \bP(V)$, with a genericaly free action of $G$, then $\varrho_\Pi$
    from \eqref{eqn:rho} is $G$-equivariant, and 
    $$
    Y\sim_G X.
    $$
    This happens, e.g., when $G$ is cyclic.  
\end{rema}

\subsection*{Equivariantly Pfaffian actions}

Finite groups which can act regularly and generically freely  on smooth cubic threefolds have been classified in \cite{weiyu}. We recall:
\begin{itemize} 
\item There are 6 maximal groups of automorphisms of smooth cubic threefolds, by \cite[Theorem 1.1]{weiyu}:
\begin{equation} 
\label{eqn:groups}
C_3^4\rtimes \fS_5, \, ((C_3^2\rtimes C_3)\rtimes C_4) \times \fS_3,\,  C_{24}, \, C_{16}, \, \mathsf{PSL}_2(\bF_{11}), \, C_3\times \fS_5.
\end{equation}
\item There are 2 types of actions of 
the Klein four-group $\mathfrak K_4$, given (in suitable coordinates) 
in \cite[Table 2]{weiyu}:
\begin{align}
    \label{eqn:0}&\mathrm{diag}(-1,1, 1, 1, 1), \quad \mathrm{diag}(1,1, -1, 1, 1),\\
    \label{eqn:1}&\mathrm{diag}(-1,-1, 1, 1, 1), \quad \mathrm{diag}(1,-1, -1, 1, 1).
\end{align}
\end{itemize}

\begin{rema}
\label{rema:klein-four}    
The unique (up to conjugation) $\mathfrak K_4\subset \mathsf{PSL}_2(\bF_{11})$ acts via \eqref{eqn:1} on the Klein cubic threefold. On the other hand, the $\mathfrak S_5$ with 
the permutation action on $\bP^4$ contains a $\mathfrak K_4$ acting via \eqref{eqn:0} on the Fermat cubic threefold. By Example~\ref{exam:incomp}, 
this second $\mathfrak K_4$ contributes an incompressible symbol to the class 
of the action in the Burnside group. 
\end{rema}

Here we investigate which groups admit equivariantly Pfaffian actions on smooth cubic threefolds. This is an algorithmic task.  With {\tt Magma}, we implement the following steps: 
\begin{enumerate}
    \item Let $\tilde G$ be one of the Schur covers of $G.$ List faithful 6-dimensional linear $\tilde G$-representations $V$ which induce generically free $G$-actions on $\bP(V)$.
    \item For each such $V$, and each isomorphism class of 5-dimensional subrepresentations of $\wedge^2(V)$,
    choose a generic such $A\subset \wedge^2(V)$; check whether or not the induced $G$-action on $\bP(A)$ is generically free and 
$$
\bP(A)\cap \Gr(2,V)^\vee
$$
is a smooth cubic threefold. 
\end{enumerate}
By construction, any equivariantly Pfaffian action on a smooth cubic threefold can be obtained in this way. Going through the steps (1) and (2), we find
\begin{prop}
\label{prop:lift}
    Let $G$ be a finite group. Then $G$ admits an equivariantly Pfaffian action on 
    a smooth cubic threefold if and only if $G$ is a subgroup of one of the following groups: $$
    \PSL_2(\bF_{11}),\quad C_3\rtimes\fD_4,\quad\fF_5,\quad C_8.
    $$
\end{prop}
\begin{proof}

We start with cyclic $G$, listed in \cite[Table 2]{weiyu}; in these cases, the Schur cover of $G$ is itself. Applying steps (1) and (2) with $\tilde G=G$,
we find that the following do not arise from this construction:
\begin{align*}
    G=C_3,\quad &\text{with weights } (1,1,1,\zeta_3,\zeta_3^a),\quad  a=0,1,2,\\
    G=C_4,\quad &\text{with weights } (1,1,1,-1,\zeta_4).
\end{align*}
Excluding all cases in \cite[Table 2]{weiyu} containing one of the $C_3$ or $C_4$ actions above, we find that the actions of the following abelian groups 
$$
C_4\times C_2, \quad C_3^2, \quad  C_9, \quad  C_{15},  \quad C_{16}, \quad  C_{18}, \quad  C_2\times C_3^2, \quad  C_{24}, 
$$
$$
C_4\times C_6, \quad  C_3^3, \quad  C_3\times C_9, \quad  C_4\times C_3^2,  \quad C_2^2\times C_3^2, \quad  C_2\times C_3^3, \quad  C_3^4
$$
on smooth $Y$ and $X$ are not equivariantly Pfaffian. 

Then excluding subgroups of the 6 maximal groups in \cite[Theorem 1.1]{weiyu} which contain a subgroup isomorphic to one of these abelian groups, we are left with 
$$
\PSL_2(\bF_{11}),\quad C_3\rtimes\fD_4,\quad \fS_5, \quad C_8,
$$
and their subgroups.

Applying steps (1) and (2) above to the Schur cover of $\fS_4$, we find that  $\fS_4$ and thus also $\fS_5$ does not admit equivariantly Pfaffian actions on {\em smooth} cubic threefolds.

Excluding $\fS_4$ and $\fS_5$, we are left with 4 maximal groups
$$
\PSL_2(\bF_{11}), \quad C_3\rtimes\fD_4, \quad \mathfrak F_5,\quad
C_8.
$$
These admit equivariantly Pfaffian actions on smooth cubic threefolds, see explicit constructions in Sections~\ref{sect:psl2}, \ref{sect:C3D4}, and \cite{TZ-web}. 
\end{proof}

Since every faithful 6-dimensional representation of $\mathfrak F_5$ and $C_8$ admits a 1-dimensional subrepresentation, the corresponding smooth cubic $Y$ and degree 14 Fano threefolds $X$ are equivariantly birational, by Remark~\ref{rema:birat}. 
Given our interest in {\em nontrivial} (twisted) stable birationalities, we present in Sections~\ref{sect:psl2} and \ref{sect:C3D4} explicit Pfaffian constructions for 
$$
\PSL_2(\bF_{11}), \quad C_3\rtimes\fD_4.
$$
\begin{rema}
    The proof of Proposition~\ref{prop:lift} confirms computations by 
B\"ohning and von Bothmer indicating that regular $C_3$-actions
on smooth cubic threefold
with weights $(1,1,1,1,\zeta_3)$ are not equivariantly Pfaffian. 
On the other hand, we will see in Section~\ref{sect:C3D4} that $C_3$-actions with weights $(1,1,\zeta_3,\zeta_3^2,\zeta_3^2)$ are equivariantly Pfaffian. Proposition~\ref{prop:lift} also shows that $\fS_5$ does not admit equivariantly Pfaffian actions on smooth cubic threefolds. However, an $\fS_5$-action on {\em singular} cubics may arise from the Pfaffian construction, see Section~\ref{sect:s5}. 
\end{rema}
\begin{rema} 
The papers \cite{CTZ-segre}, \cite{CTZ}, \cite{CMTZ} classify actions on cubic threefolds with isolated singularities, under the assumption that the action does not fix any of the singular points, and apply this classification to linearizability questions. It would be interesting to explore the equivariant Pfaffian-Grassmannian correspondence for singular cubic threefolds.  
\end{rema}

\section{Klein cubic threefold and $\PSL_2(\bF_{11})$-actions}
\label{sect:psl2}

Here we apply the equivariant Pfaffian-Grassmannian correspondence of Section~\ref{sect:dual} to the Klein cubic threefold and the associated degree 14 Fano threefold, equipped with the action of 
$
G=\PSL_2(\bF_{11}).
$

\subsection*{Writing the representation}
Let $V$ be one of the two irreducible $6$-dimensional representations of $$
\tilde G:=\SL_2(\bF_{11}).
$$
Assume $V$ has character 
$$
\mathrm{char}(V)=( 6, -6, 0, 0, 1, 1, 0, -1, -1, -\lambda, \lambda + 1, 0, 0, -\lambda-1, \lambda),
$$
$$
 \lambda:= \zeta_{11}^9 + 
    \zeta_{11}^5 + \zeta_{11}^4 + \zeta_{11}^3 + \zeta_{11}.
$$
Then 
$$
\wedge^2(V^\vee)=A\oplus V_{10},
$$
where $A$ and the dual $V_{10}^\vee= A^\perp\subset \wedge^2(V)$ are
faithful irreducible representations of $G$ of dimension 
$5$, respectively, $10$. The Pfaffian cubic 
$$
Y:=\bP(A)\cap \Gr(2,V)^\vee
$$
is a smooth cubic threefold with a generically free action of $G$. Such a cubic (the Klein cubic threefold) is unique; up to a change of variables, it is given by the equation 
$$
\{y_1^2y_2+y_2^2y_3+y_3^2y_4+y_4^2y_5+y_5^2y_1=0\}\subset\bP^4_{y_1,\ldots,y_5}. 
$$
The dual Fano threefold 
$$
X:= \Gr(2,V)\cap\bP(A^\perp)\subset \bP^9
$$ 
of degree 14 is also smooth. The equations can be found at \cite{TZ-web}.

\subsection*{Stabilizer stratification}

We compute the fixed loci stratification of the $G$-action on $Y$ and $X$ with {\tt Magma}, recording the data for (orbit representatives of) 
loci $F$ with nontrivial generic stabilizeres:
\begin{itemize}
    \item stabilizer of $F$,
    \item the residual action on $F$,
    \item dimension of $F$,
    \item degree of $F$,
    \item characters of the induced action on the normal bundle to $F$.
\end{itemize}

The fixed loci stratification for the $G$-action on $Y$ is given by
$$
\begin{tabular}{c|c|c|c|c|c}

&Stabilizer&Residue&dim&deg&Characters\\
\hline
 1&   $C_6$& {triv}&  $0$&  1& $( 4, 1, 5)$\\\hline
2--3&   $C_2^2$&   {triv}&  $0$&  $1$& $(( 1, 0), (0, 1 ), ( 1, 1 ))$\\\hline
4&   $C_{11}$&   {triv}&  $0$&  $1$& $( 2, 3 , 4)$\\\hline
5&   $C_5$&  {triv}&  0&  1& $( 1 , 3 , 4)$\\\hline
6&   $C_5$&   {triv}&  0&  1& $( 3,  2 , 1)$\\\hline
7&   $C_3$&   {triv}&  $0$&  1& $(2, 1 , 2 )$\\\hline
8&   $C_2$&  $\fS_3$&  1&  3& $( 1 ,  1 )$\\\hline
9&   $C_2$&  $\fS_3$&  1&  1& $( 1, 1)$ \\
\end{tabular}
$$

\ 

On the smooth degree 14 Fano threefold $X$, it is:
$$
\begin{tabular}{c|c|c|c|c|c}
&Stabilizer&Residue&dim&deg&Characters\\
\hline
1&   $\fA_4$&  triv&  0&  1&N/A\\\hline
2& $\fS_3$&  triv&  0& 1&N/A\\\hline
3&   $C_6$&  triv&  0&  1& $ (3,1,5)$\\\hline
4&   $\fS_3$& triv&  0&  1&N/A\\\hline
5&   $C_2^2$&  triv&  0&  1& $(( 1, 1 ), ( 0, 1 ), ( 1, 0 ))$\\\hline
6&   $C_{11}$&  triv&  0& 1& $( 3,  10, 5)$\\\hline
7--8&   $C_5$&  triv&  0&  1& $(3 , 1, 2)$\\\hline
9&   $C_3$& triv&  0&  1& $( 2, 2, 1)$\\\hline
10&   $C_3$&  $C_2^2$&  1&  2& $(2, 1)$\\\hline
11&   $C_2$&  $\fS_3$&  1&  6& $( 1, 1)$\\
\end{tabular}
$$

\subsection*{Burnside invariants}

The vanishing relation in $\Burn_3(G)$ mentioned in Section~\ref{sect:birgeo} implies that the {\em only} nontrivial contribution to 
$[Y\actsfromright G]$ comes from points with stabilizer $C_{11}$; and similarly, for $X$. (This is justified even though both $Y$ and $X$ are not standard models.)
A direct computation shows that 
\begin{align*} 
[Y\actsfromright G] & = (C_{11}, 1\actsfromleft k, (2,3,4))\\
& = (C_{11}, 1\actsfromleft k, (3,10,5)) = [X\actsfromright G]\in \Burn_3(G).
\end{align*}
This is not surprising, given Remark~\ref{rema:birat}: $Y$ and $X$ are $C_{11}$-equivariantly birational.
In particular, the Burnside formalism does not allow to distinguish these actions. 

\subsection*{Twisted equivariant stable birationality}
Birational rigidity techniques yield (see \cite[Theorem 4.3]{BC-finite})
$$
Y\not\sim_G X. 
$$
On the other hand, 
applying Proposition~\ref{prop:stable}, we have a twisted $G$-equivariant stable birationality:

\begin{prop} \label{prop:psl211}
Let $G=\PSL_2(\bF_{11})$, acting on 
the Klein cubic threefold $Y$ and the associated degree 14 Fano threefold $X$. 
Then
$$
Y\times \bP^2\times \bP(V) \sim_G X\times \bP^2\times \bP(V),
$$
with trivial $G$-action on $\bP^2$ and
projectively linear $G$-action on $\bP(V)$, arising from a 6-dimensional irreducible representation $V$ of $\SL_2(\bF_{11})$. 
\end{prop}

\section{$C_3\rtimes \fD_4$-actions}
\label{sect:C3D4}

We provide additional examples of nonbirational, but twisted equivariantly stably birational, actions on smooth $Y$ and $X$. Here, the nonbirationality of the actions is established via Burnside invariants. 

\subsection*{Writing the representation}

Let $G=C_3\rtimes \fD_4$ and $V$ be
a faithful 6-dimensional representation  of $\tilde{G}\simeq C_3\rtimes\fD_8$, decomposing as
$$
V=V_2\oplus V_4, 
$$
where $V_2$ and $V_4$ are irreducible $\tilde G$-representations with character
\begin{align}\label{eqn:V4}
 \notag    \mathrm{char}(V_2)&=( 2, -2, 0, 0, 2, 0, -2, 0, 0, \lambda, -\lambda, 0 ),\quad\lambda=\zeta_8^3 - \zeta_8,\\
    \mathrm{char}(V_4)&=( 4, -4, 0, 0, -2, 0, 2, 0, 0, 0, 0, 0 ).
\end{align}
Then 
\begin{align}\label{eq:decomDihedralW2smooth}
    \wedge^2(V^\vee)=U_1^{\oplus2}\oplus U_2\oplus W_1^{\oplus2}\oplus W_2\oplus W_3\oplus W_4\oplus W_5,
\end{align}
where $U_1$ and $U_2$ are distinct $1$-dimensional and $W_i,i=1,\ldots,5$ distinct $2$-dimensional representations of $G.$


Let $A\subset\wedge^2(V^\vee)$  be a generic $5$-dimensional subspace 
isomorphic, as a $G$-representation to 
$$
U_1\oplus W_1\oplus W_2.
$$
One can check that the Pfaffian cubic 
$$
Y:=\bP(A)\cap \Gr(2,V)^\vee
$$
and the associated degree 14 Fano threefold
$$
X:=\bP(A^\perp)\cap \Gr(2,V)
$$
are smooth, and carry a generically free action of $G$; see \cite{TZ-web} for an example and equations.

\subsection*{Stabilizer stratification}

Note that $G$ contains {\em two} conjugacy classes of the Klein four-group $\mathfrak K_4$, corresponding to the two $\mathfrak K_4$-actions \eqref{eqn:0} and \eqref{eqn:1}, with respective characters in $A$ given by
\[
    (5,-3,1,-3),\quad (5,-1,1,-1).
\]
The fixed loci stratification for the first $\mathfrak K_4$-action
is given by 
$$
\begin{tabular}{c|c|c|c|c|c}

&Stabilizer&Residue&dim&deg&Characters\\
\hline
 1&   $\mathfrak K_4$& {triv}&  $1$&  3& $( (1,1),(1,0))$\\\hline
2&   $\mathfrak K_4$& {triv}&  $0$&  1& $( (1,0),(0,1),(1,0))$\\\hline
 3&   $\mathfrak K_4$& {triv}&  $0$&  1&$( (1,1),(0,1),(1,1))$\\\hline
  4&   $C_2$& $C_2$&  $2$&  3&$( 1)$\\\hline
    5&   $C_2$& $C_2$&  $2$&  3&$( 1)$\\\hline
    6&   $C_2$& $C_2$&  $1$&  1&$( 1,1)$\\  
\end{tabular}
$$
The fourth and fifth strata are smooth cubic surfaces, and the first stratum is a smooth cubic curve contained in both cubic surfaces.

\

On the smooth degree 14 Fano threefold $X$, we have:
$$
\begin{tabular}{c|c|c|c|c|c}

&Stabilizer&Residue&dim&deg&Characters\\
\hline
 1--7&   $C_2$& {triv}&  $0$&  1& $( 1,1,1)$\\\hline
   8&   $C_2$& {triv}&  $1$&  1& $( 1,1)$\\\hline
   9&   $C_2$& {triv}&  $1$&  1& $( 1,1)$\\\hline
     10&   $C_2$& $C_2$&  $1$&  6& $( 1,1)$\\
\end{tabular}
$$
The last stratum is a degree 6 smooth curve of genus $1.$ 
\

\subsection*{Burnside invariants}

Using the stabilizer stratification above, we obtain:

\begin{prop}
\label{prop:burn-c3}
For all $G'\subseteq G$ containing a conjugate of $\mathfrak K_4\subset G$ which acts on $A$ with character
$$
(5,-3,1,-3),
$$
we have
$$
Y\not\sim_{G'} X. 
$$
\end{prop}

\begin{proof}
To establish nonbirationality, consider the specified action of $\mathfrak K_4\subset G$ on $Y$. In the stabilizer stratification, 
we find incompressible symbols 
$$
(C_2,C_2\actsfromleft k(S), (1)),
$$
where $S$ is a cubic surface and the 
residual $C_2$-action fixes a smooth cubic curve on $S.$
This contributes to the class 
$$
[Y\actsfromright \mathfrak K_4] \in \Burn_3(\mathfrak K_4),
$$
and is an instance of Example~\ref{exam:incomp}. 
On the other hand, we see from the stabilizer stratification for $X$ that no such symbols arise in $[X\actsfromright \mathfrak K_4]$. This implies that 
$$
Y\not\sim_{\mathfrak K_4} X.
$$    
\end{proof}

\subsection*{Twisted equivariant stable birationality}

Applying Proposition~\ref{prop:stable} and the construction there with $W=V_4$, we obtain: 

\begin{prop}
\label{prop:c3}
We have
$$
Y\times \bP^2\times \bP(V_4)\sim_{G} X\times \bP^2\times \bP(V_4),
$$
with trivial $G$-action on $\bP^2$ and projectively linear $G$-action on 
$\bP(V_4)$, arising from the faithful  $\tilde{G}$-representation 
$V_4$  given by \eqref{eqn:V4}. 
\end{prop}

\section{Singular examples}
\label{sect:cubic-sing}

Here, we consider actions of dihedral groups $\mathfrak D_{2n}$ of order 
$4n$ on cubic threefolds with 2$\sD_4$-singularities.

\subsection*{Writing the representation}

We start with a faithful irreducible $2$-dimensional representation $V_2=V_2(\chi)$ of $\fD_{4n}$, determined by a primitive character $\chi=\chi_{4n}$ of $C_{4n}$. In detail, choose generators
$$
\fD_{4n}=\langle s,t\mid s^{4n}=t^2=tsts=1\rangle,
$$
and consider the representation $V_2$ given by 
$$
s\mapsto\begin{pmatrix}
  \zeta_{4n}^{-r}&0\\
    0&\zeta_{4n}^{r}
\end{pmatrix},\quad t\mapsto\begin{pmatrix}
   0&1\\
    1&0
\end{pmatrix},
$$
where $r$ is determined by $\chi_{4n}$, $\gcd(r,4n)=1$. Put
$$
V=V_2^{\oplus 3},
$$
this is a faithful 6-dimensional representation of $\fD_{4n}$, with 
a generically free action of $\fD_{2n}$ on its projectivization
$\bP(\wedge^2 V)$. 

\subsection*{Decomposing the representations}
We have:
\begin{align}\label{eq:decomDihedralW2}
    \wedge^2(V)=I^{\oplus3}\oplus W_1^{\oplus 6}\oplus W_2^{\oplus 3},
\end{align}
where 
\begin{itemize} 
\item $I$ is the trivial representation of $\fD_{2n}$,
\item $W_2$ is the irreducible 2-dimensional representation of $\fD_{2n}$ determined by the character $\chi_{4n}^2$, and 
\item $W_1$ is the 1-dimensional representation of $\fD_{2n}$ given by 
$$
s\mapsto (1),\quad t\mapsto (-1).
$$
\end{itemize}
Let $z_1,\ldots,z_6$ be coordinates of $V,$ and fix the following basis of $\wedge^2(V)$:
\begin{multline*}
    u_1=z_1\wedge z_2,\quad u_2=z_1\wedge z_3,\quad\ldots,\quad u_5=z_1\wedge z_6,\quad
    u_6=z_2\wedge z_3,\\ u_7=z_2\wedge z_4,\quad \ldots, \quad u_{15}=z_5\wedge z_6.
\end{multline*}
We can rewrite the decomposition in an appropriate basis
\begin{align*}
    I^{\oplus3}&=\langle u_3+u_6\rangle\oplus \langle u_5+u_8\rangle\oplus \langle u_{12}+u_{13}\rangle,\\
     W_1^{\oplus6}&=\langle u_1\rangle\oplus\langle u_{10}\rangle\oplus\langle u_{15}\rangle\oplus\langle u_3-u_6\rangle\oplus \langle u_5-u_8\rangle\oplus \langle u_{12}-u_{13}\rangle,\\
     W_2^{\oplus 3}&=\langle u_2, u_7\rangle\oplus\langle u_4, u_9\rangle\oplus\langle u_{11}, u_{14}\rangle.
\end{align*}

\subsection*{The invariant Pfaffian cubic}
Choose a basis $v_1,\ldots,v_{15}$
of $\wedge^2(V^\vee)$ corresponding to $u_1,\ldots,u_{15}$, we have a similar decomposition of $\wedge^2(V^\vee)$ as in \eqref{eq:decomDihedralW2}.
Fix the $5$-dimensional subspace $A\subset\wedge^2(V^\vee)$, with basis
\begin{align}
\label{eqn:A}
       y_1&=v_2,\\
y_2&=v_7,\nonumber\\
y_3&=v_5 -v_8-v_{10} + v_{12} -v_{13}-v_{15},\nonumber\\
y_4&=-12v_1 + v_5 -v_8 + 2v_{10} + v_{12} -v_{13},\nonumber\\
y_5&=-3v_1 + v_5 -v_8 -v_{10}.\nonumber
\end{align}
Then $\fD_{2n}$ acts generically freely on $\bP(A)=\bP^4_{y_1,\ldots,y_5}$, via 
\begin{align*}
    s&: \mathbf{(y)}\mapsto (\zeta_{2n}^{-r}y_1,\zeta_{2n}^{r}y_2,y_3,y_4,y_5),\\
      t&: \mathbf{(y)}\mapsto (y_2,y_1,-y_3,-y_4,-y_5).
\end{align*}
This induces a $\fD_{2n}$-action on
$$
Y:=\bP(A)\cap\Gr(2,V)^\vee,
$$ 
a cubic threefold with $2\sD_4$-singularities given  by 
\begin{equation} 
\label{eqn:yyy}
Y=\{y_1y_2y_3 + y_3^3 + 3y_3^2y_5 + 45y_3y_4^2 + 10y_4^3 + y_5^3=0\}\subset\bP^4_{y_1,\ldots,y_5}.
\end{equation}

\subsection*{The invariant Fano threefold} 
The annihilator $A^\perp\subset \wedge^2(V)$ of $A$ is a $10$-dimensional subspace, with basis
\begin{align*}
&x_1=u_{1} - 3u_{10} +9u_{12}- 9u_{13} + 21u_{15},\quad &x_2=u_{3},\quad &x_3=u_{4},\\
&x_4=u_{5}-u_8 + 2u_{10} -3u_{12}+ 3u_{13} - 6u_{15},\quad &x_5=u_{6},\quad & \\
&x_6=u_5+u_8,\quad &x_7=u_{9},\quad &x_8= u_{11},
\\
&x_9= u_{12} + u_{13}, \quad &x_{10}=u_{14}.\quad &  
\end{align*}
The degree 14 Fano threefold 
$$
X=\Gr(2,V)\cap\bP(A^\perp) \subset\bP^9_{x_1,\ldots,x_{10}}
$$ 
is given by 
the vanishing of the following polynomials:
\begin{align*}
 &-144x_{1}^2 + 114x_{1}x_{4} - 21x_{4}^2 - x_{8}x_{10} + x_{9}^2,\\
&    3x_{1}x_{4} - 9x_{1}x_{5} - 3x_{1}x_{6} - 2x_{4}^2 + 3x_{4}x_{5} + 2x_{4}x_{6} + x_{5}x_{9},\\
&    -9x_{1}x_{7} + 3x_{4}x_{7} + x_{4}x_{10} - x_{6}x_{10} + x_{7}x_{9},\\
&    21x_{1}x_{2} - 9x_{1}x_{4} - 9x_{1}x_{6} - 6x_{2}x_{4} - x_{3}x_{10} + 3x_{4}^2 + 3x_{4}x_{6} + x_{4}x_{9} +
        x_{6}x_{9},\\
&    -9x_{1}x_{3} + 3x_{3}x_{4} - x_{3}x_{9} + x_{4}x_{8} + x_{6}x_{8}\\
&    x_{1}x_{8} + x_{3}x_{5},\\
&    -9x_{1}^2 + 3x_{1}x_{4} + x_{1}x_{9} + x_{2}x_{4} - x_{2}x_{6},\\
&    -3x_{1}^2 + 2x_{1}x_{4} + x_{2}x_{5},\\
&    x_{1}x_{10} - x_{2}x_{7},\\
&    -9x_{1}x_{2} - 3x_{1}x_{4} - 3x_{1}x_{6} + 3x_{2}x_{4} - x_{2}x_{9} + 2x_{4}^2 + 2x_{4}x_{6},\\
&    -3x_{1}x_{7} + 2x_{4}x_{7} + x_{5}x_{10},\\
&    9x_{1}^2 - 3x_{1}x_{4} + x_{1}x_{9} + x_{4}x_{5} + x_{5}x_{6},\\
&    21x_{1}^2 - 6x_{1}x_{4} - x_{3}x_{7} - x_{4}^2 + x_{6}^2,\\
&    -3x_{1}x_{3} - x_{2}x_{8} + 2x_{3}x_{4},\\
&    9x_{1}x_{4} + 21x_{1}x_{5} - 9x_{1}x_{6} - 3x_{4}^2 - 6x_{4}x_{5} + 3x_{4}x_{6} + x_{4}x_{9} - x_{6}x_{9} + 
        x_{7}x_{8}.
\end{align*}
The group $\fD_{2n}$ acts generically freely on $\bP^9$ and $X$, via
\begin{align*}
    s&:\mathbf{(x)}\mapsto(x_1,x_2,\zeta^{-r}x_3,x_4,x_5,x_6,\zeta^{r}x_7,\zeta^{-r}x_8,x_9,\zeta^{r}x_{10}),\quad \zeta=\zeta_{2n},\\
    t&:\mathbf{(x)}\mapsto(-x_1,x_5,x_7,-x_4,x_2,x_6,x_3,x_{10},x_9,x_8).
\end{align*}
The Fano threefold $X$ is {\em singular} along two disjoint {\em lines}
$$
\{x_1=x_2=x_3=x_4=x_5=x_6=x_8=x_9=0\}
$$
and
$$\{x_{1}=
x_{2}=
x_{4}=
x_{5}=
x_{6}=
x_{7}=
x_{9}=
x_{10}=0
\}.
$$

\begin{rema}
\label{rema:singu}
    Note that here  $\Sing(X)\not\simeq\Sing(Y)$, contrary to \cite[Proposition A.4]{KuzCubic}. This is explained by the fact that the embedding $f:A\hookrightarrow\wedge^2 (V^\vee)$ is {\em not} regular, which requires that {\em all} forms in $A$ have rank $\ge 4$. In our case, 
    for general subrepresentations in $\wedge^2 (V^\vee)$ isomorphic to $A$, from \eqref{eqn:A}, the locus $\bP(A)\cap\Gr(2,V^\vee)$ consists of two {\em points}, 
    representing skew-symmetric forms of rank $2$ in $\wedge^2 (V^\vee)$. These two points are the singular points of the cubic hypersurface $Y$. 
\end{rema}

\subsection*{Stabilizer stratification}

We record the stabilizer stratification for the $G$-action on the singular cubic threefold $Y$:

$$
\begin{tabular}{c|c|c|c|c|c}

&Stabilizer&Residue&dim&deg&Character\\
\hline
 1&   $\fD_{2n}$& {triv}&  $1$&  3& N/A\\\hline
2&   $C_n$&   {triv}&  $0$&  $1$& Singular point\\\hline
3&   $C_2^2$&  {triv}&  $0$&  1& $(( 0, 1 ), ( 1, 0 ), ( 0,1 ))$\\\hline
4&   $C_2^2$&  {triv}&  $0$&  1& $(( 0, 1 ), ( 1, 0 ), ( 0,1 ))$\\\hline
5&   $C_2$&  $C_2$&  $2$&  $3$& $( 1)$\\\hline
6&   $C_2$&   $C_2$&  $2$&  $3$& $( 1)$\\\hline
7&   $C_2$&  $C_n$&  $1$&  1& $(1,1)$\\
\end{tabular}
$$

\ 

As in Section~\ref{sect:C3D4}, the strata 5 and 6 are cubic surfaces, with residual action fixing the same smooth cubic curve $E_Y$ in the stratum 1. 

\ 

The stratification on the degree 14 Fano threefold $X$ is:

$$
\begin{tabular}{c|c|c|c|c|c}

&Stabilizer&Residue&dim&deg&Character\\
\hline
 1&   $C_n$& {triv}&  $1$&  1& Singular line\\\hline
2&   $C_n$&   $C_2$&  $1$&  $6$& $(2r,-2r)$\\\hline
3--8&   $C_2$&  triv&  $0$&  1& $(1,1,1)$\\\hline
9&   $C_2$&  triv&  $1$&  $1$& $(1,1)$\\\hline
10&   $C_2$&  triv&  $1$&  $1$& $(1,1)$\\
\end{tabular}
$$

\subsection*{Burnside invariants}

The first stratum on $Y$ has generic stabilizer $\fD_{2n}$, it is a smooth cubic curve 
$$
E_Y:=\{ y_1=y_2=0 \} \cap Y.
$$
We are in the situation of Example~\ref{exam:dih}: 
the model $Y$ is not in standard form; after blowing up $E_Y$, we obtain the incompressible symbol
\begin{equation} 
\label{eqn:symbb}
(C_2, \fD_{n}\actsfromleft k(E_Y)(t), (1)).
\end{equation}
Note that the action on the generic fiber of the projectivization of the normal bundle to $E_Y$ is nonabelian.

On the other side, the second stratum in the stabilizer stratification of the 
degree 14 Fano threefold $X$ is a degree 6, smooth genus 1 curve 
$$
E_X:=\{x_3=x_7=x_8=x_{10}=0\}\cap X,
$$
with stabilizer $\langle s\rangle\simeq C_n$. We checked, via {\tt Magma}, that $E_Y$ and $E_X$ have the same $j$-invariant and thus are isomorphic. 

We know that the $G$-action on $X$ is not in standard form. However, even after blowups, the resulting class in $\Burn_3(G)$
differs from the symbol \eqref{eqn:symbb} -- we cannot get a nonabelian action on the fibers of a $\bP^1$-bundle over $E_X$, from an abelian stabilizer. 

Therefore, the $\fD_{2n}$-actions on $X$ and $Y$ are not equivariantly birational.

\section{$\fS_5$-actions}
\label{sect:s5}

Here we investigate the Pfaffian construction for $
G=\fS_5.
$
Consider the central extension 
$$
\tilde G=C_2\rtimes\SL_2(\bF_5)
$$ 
with GAP ID (240,90). Note that $\tilde G$ is one of the Schur covers of $\fS_5$ (the other central extensions yield a picture similar to the one described below). 

The group $\tilde{G}$ has two faithful irreducible 6-dimensional linear representations, differing by the sign character. Let $V$ be the one with character 
$$
(6,-6, 0,0,0,1,0,0,0,-\zeta^3-\zeta,\zeta^3+\zeta,-1),\quad \zeta=\zeta_8.
$$
Then $\wedge^2(V)$ is a faithful $G$-representation decomposing as
$$
\wedge^2(V^\vee)=V_1\oplus V_4\oplus V_5\oplus V_5',
$$
where the character of each summand is given by 
\begin{align*}
    V_1&: ( 1, -1, 1,  1, -1, 1, -1 ),\\
    V_4&: ( 4, -2, 0,  1, 0,  -1,1 ),\\
    V_5&: ( 5, -1, 1, -1, 1, 0,-1 ),\\
    V_5'&:( 5, 1,1,-1,-1,0,1 ).
\end{align*}
There are three distinct 5-dimensional subspaces $A\subset \wedge^2(V^\vee)$, namely $V_5$, $V_5'$ and  $V_1\oplus V_4$ . The resulting Pfaffian cubic threefolds
$$
Y=\bP(A)\ \cap \ \Gr(2,V)^\vee
$$
have different singularity types. 

\subsection*{$A=V_5$} The cubic threefold  $Y$ has 6 $\sA_1$-singularities. The $\fS_5$-action on $Y$ is unique, see \cite[Proposition 7.3]{CTZ}.

\subsection*{$A=V_5'$} The cubic $Y$ 
is the Segre cubic threefold, the unique cubic with 10 $\sA_1$-singularities. 
The $\fS_5$-action with the prescribed character is unique. It is the {\em nonstandard} $\fS_5$ in $\Aut(Y)=\fS_6$, and is linearizable, see \cite[Section 6]{CTZ}.

\subsection*{$A=V_1\oplus V_4$} The cubic threefold $Y$ has 5 $\sA_1$-singularities, with a transitive action of $\fS_5$. Such an $\fS_5$-action is unique \cite[Section 6]{CTZ}. By \cite[Theorem 3.1]{VAZ}, the only $G$-Mori fiber spaces which are $G$-birational to $Y$ are a smooth quadric threefold and $Y$ itself; under the standard Cremona involution,  $Y$ is $\fS_5$-equivariantly birationally transformed to the smooth quadric threefold given by 
\begin{multline*}
    y_1y_2 + y_1y_3 + y_2y_3 + y_1y_4 + y_2y_4 + y_3y_4 + y_1y_5 + y_2y_5 + y_3y_5 + y_4y_5=0,
\end{multline*}
with the same $\fS_5$ permutation action on the coordinates.
The singular locus of the dual Fano threefold $X$ consists of $10$ points, i.e., for each singularity on $Y$ there {\em two} singular points on $X$ -- the corresponding $A$-net is not regular.

\

\bibliographystyle{alpha}
\bibliography{pfaff-cubic}

\end{document}